\documentclass[11pt,a4paper]{article}
\usepackage[numbers,sort&compress]{natbib}
\usepackage[colorlinks,
            linkcolor=blue,
            anchorcolor=green,
            citecolor=magenta
            ]{hyperref}
\usepackage{mathrsfs,amssymb,amsfonts}

\usepackage{graphicx}
\usepackage{amsmath,amssymb,latexsym,color}
\usepackage[mathscr]{eucal}
\usepackage[numbers]{natbib}
\usepackage{bm}
\usepackage{subfigure}
\usepackage[top=1in, bottom=1in, left=1.25in, right=1.25in]{geometry}
\usepackage{indentfirst}
\usepackage{array}
\usepackage{graphicx}
\usepackage{setspace}
\usepackage{epstopdf}
\usepackage{cleveref}

\usepackage{float}

\newtheorem{theorem}{Theorem}[section]
\newtheorem{definition}[theorem]{Definition}

\newtheorem{lemma}[theorem]{Lemma}
\newtheorem{remark}[theorem]{Remark}

\newtheorem{example}{Example}

\newcommand{\qed}{\hfill\Box\medskip}
\usepackage{CJK}
\usepackage{geometry}
\begin{document}

\setlength{\baselineskip}{14pt}
\renewcommand{\abovewithdelims}[2]{
\genfrac{[}{]}{0pt}{}{#1}{#2}}

\title{\bf {The cyclic diagnosability of Cayley graphs generated by transposition triangle-free unicyclic graphs under the PMC and MM$^*$ models}\footnote{This research is supported by the National Natural Science Foundation of China(12301453, 12271524, 12331013), the Project of Education Department of Hunan Province(23B0125), the Project of Scientific Research Fund of Hunan Provincial Science and Technology Department (No. 2018WK4006).}}

\author{ Jialu Ding\quad Caixia Li\footnote{Corresponding author. \newline {\em E-mail address:} 2932892417@qq.com(J. Ding), caixial@yeah.net(C. Li), lips@xtu.edu.cn (P. Li), jinweipei82@163.co(W. Jin)}\quad Pingshan Li \quad Wei Jin \\
{\footnotesize  School of Mathematics and Computational Science, Xiangtan University, Xiangtan, Hunan 411105, PR China}\\
{\footnotesize Key Laboratory of Intelligent Computing$\And$Information Processing of  Ministry of Education}\\
{\footnotesize Key Laboratory for Computation and Simulation in Science and Engineering}\\
{\footnotesize National Center for Applied Mathematics in Hunan}\\
}

\date{}
 \maketitle
\begin{abstract}
\setlength{\parindent}{2em}
Diagnosability is a critical parameter for evaluating the reliability and self-diagnostic capacity of multiprocessor systems. As an advanced extension of traditional diagnosability, cyclic diagnosability has been proposed to enable more comprehensive assessment of the diagnostic capabilities of interconnection networks. In this paper, we prove that under both the PMC and MM$^*$ models, the cyclic diagnosability of Cayley graphs generated by triangle-free unicyclic graphs is $5n-10$ for $n \geq 11$.


\medskip
\end{abstract}

\section{Introduction}\label{section1}
\setlength{\parindent}{2em}

In large-scale parallel computing environments, multiprocessor systems are widely applied in fields such as scientific computing, data analysis, and artificial intelligence due to their high performance and reliability requirements. As the system scale continues to expand, the probability of processor node failures increases significantly. Therefore, fault diagnosis capability has become one of the important indicators for evaluating the reliability of multiprocessor systems. Typically, a multiprocessor system can be abstracted as an undirected graph, where vertices represent processors and edges represent the testing or communication relationships between processors.

The core of system-level fault diagnosis lies in identifying faulty processors through appropriate testing mechanisms. To this end, researchers have proposed various system-level diagnostic testing models. Preparata, Metze, and Chien \cite{preparata2006connection} proposed the PMC model, which is one of the earliest system-level diagnosis models and also one of the most widely used classical models. Maeng and Malek \cite{malek1980comparison,Maeng1981comparison} proposed the comparison model, namely the MM model. Subsequently, Sengupta and Dahbura \cite{sengupta1992self} improved the MM model and proposed the MM$^*$ model, which serves as an important complement to address the limitations of the PMC model.

The traditional diagnosability of multiprocessor systems focuses only on the maximum number of faulty processors that can be identified, without considering the structure of the remaining network after faults occur. To better characterize the fault tolerance capability of systems under specific fault scenarios, researchers have introduced various generalized notions of diagnosability. 
Lai et al. \cite{lai2005conditional} proposed the concept of conditional diagnosability by restricting that any faulty set cannot contain all the neighbors of any vertex in the graph. Peng et al. \cite{peng2012g} introduced the $g$-good-neighbor diagnosability, which requires that each vertex in the system has at least $g$ fault-free neighbors. Zhang and Yang \cite{zhang2016g} proposed the $g$-extra diagnosability, with the requirement that each remaining component contains at least $g + 1$ vertices. 
Subsequently, Zhang et al. \cite{zhang2022component} further proposed the $g$-component diagnosability, which demands that the remaining graph contains at least $g$ components. 
Based on these studies on diagnosability, Zhang et al. \cite{zhang2023characterization} introduced the concept of cyclic diagnosability, which requires that after removing faulty vertices, at least two remaining components each contain a cycle.

At present, significant progress has been made in the study of cyclic diagnosability. For instance, Zhang et al. \cite{zhang2023characterization} established a metric relationship between cyclic diagnosability and cyclic connectivity in regular networks, and  accordingly determined the cyclic diagnosability of the hypercube $Q_n$, the locally twisted cube $LTQ_n$, and the alternating group network $AN_n$. Up to now, researchers have determined the cyclic diagnosability of the $n$-dimensional balanced hypercube $BH_n$, the $k$-dimensional DCell network $D_{k,n}$, the folded hypercube $FQ_n$, the star graph $S_n$, and the Cayley graphs generated by transposition trees $\mathcal{T}_n$ under the PMC and MM$^*$ models ; see \cite{han2024cyclic,zheng2025cyclic,wang2025cyclic,mei2025cyclic,ZHENG2025109} for more details.

Cayley graphs generated by transposition unicyclic graphs are vertex-transitive, with excellent symmetry and a small diameter. They can be flexibly expanded by increasing the number of vertices $n$, making them suitable for constructing large-scale multiprocessor systems. Numerous properties of such graphs have been extensively studied. For example, Wang and Zhou \cite{yanna2022generalized} proved that the generalized $3$-connectivity of these Cayley graphs is $n-1$ when $n\ge4$. Gu et al. \cite{gu2021conditional} investigated the conditional diagnosability of such graphs under the PMC and MM$^*$ models. Tian et al.\cite{tian2025reliability} showed that the cyclic connectivity of Cayley graphs generated by triangle-free unicyclic graphs is $4n-8$. 

In this paper, we investigate the cyclic diagnosability of Cayley graphs generated by triangle-free unicyclic graphs under the PMC and MM$^*$ models. We conclude that the cyclic diagnosability of Cayley graphs generated by triangle-free unicyclic graphs is $5n-10$ when $n\ge11$.

The remainder of this paper is organized as follows. Section $2$ introduces the preliminaries. Section $3$ analyzes the structure of the cyclic diagnosability of Cayley graphs generated by triangle-free unicyclic graphs and determines its cyclic diagnosability under the PMC and MM$^*$ models. Section $4$ concludes the paper.
\section{Preliminary}
\label{section2}

We first introduce some notation and results that will be used throughout the paper.
Let $G=(V(G),E(G))$ be a simple, finite and undirected graph, where $V(G)$ is its vertex set and $E(G)$ is its edge set. For any vertex $v \in V(G)$,  $N_G(v)=\{u\in V(G)|(u,v)\in E(G)\}$ is the set of all vertices adjacent to \(v\) in \(G\), and the cardinality of \(N_G(v)\) is the degree of \(v\), denoted by \(d_G(v)\). A regular graph is a graph where each vertex has the same degree and a regular graph with vertices of degree $k$ is called a $k$-regular graph. For a subset \( S \subseteq V(G) \), we define the {open neighborhood} of \( S \) as 
\[N_G(S) = \bigcup_{u \in S} N_G(u) \setminus S.\]
Additionally, we denote by \( G[S] \) the {induced subgraph} of \( S \), where 
\[
V(G[S]) = S \quad \text{and} \quad E(G[S]) = \left\{ (u,v) \in E(G) \mid u, v \in S \right\}.
\]
The complement of $S$ is defined as $\overline{S} =V(G)\setminus S=\left \{ v\in V(G)|v \notin S \right \} $. We denote by \(G - S\) the subgraph obtained from \(G\) by deleting the vertices in \(S\) together with their incident edges. 

Let \(C = \left \langle   u_1, u_2, \cdots, u_k\right \rangle\) be a cycle in which \(u_i\) is exactly adjacent to \(u_{i-1}\) and \(u_{i+1}\) for \(2 \leq i \leq k-1\) and \(u_1\) is adjacent to \(u_k\). We define the length of a cycle to be the number of its edges. A cycle with \(k\) edges is called a \(k\)-cycle. The length of a shortest cycle in \(G\) is called the girth of \(G\), denoted by \(g(G)\). The unicyclic graph is a connected graph with exactly one cycle. 
A matching \(M\) in a graph \(G\) is a set of pairwise nonadjacent edges; \(M\) is a perfect matching if and only if \(M\) covers every vertex of the graph \(G\). The symmetric difference between vertex sets \(F_1\) and \(F_2\) is defined as \(F_1 \triangle F_2 = (F_1 \cup F_2) \setminus (F_1 \cap F_2)\). We denote by $odd(G)$ the number of components with $odd$ vertices of $G$. 
\subsection{Two classical diagnostic models}
Next, we introduce two widely recognized diagnostic models for system-level diagnosis: the PMC model and the MM$^*$ model.

Let $F \subseteq V(G)$ denote the set of faulty vertices. For edges $(u, v), (u, w) \in E(G)$, the test rules of the two diagnostic models are defined as follows.

Under the PMC model, adjacent vertices can test one another. If vertex $u$ tests vertex $v$, then $u$ is referred to as the tester and $v$ as the testee. The outcome of this test is denoted by $\sigma(u, v)$, and its value is determined by the following rule: 
$$
\sigma(u, v) =
\begin{cases}
0, & \text{if } \{v, u\} \cap F = \emptyset; \\
1, & \text{if } u \notin F \text{ and } v \in F; \\
0 \text{ or } 1, & \text{if } u \in F,
\end{cases}
$$
where the test outcome is reliable only if the tester is fault-free (i.e., $u \notin F$).

Under the MM$^*$ model, a vertex can test two of its neighbors simultaneously. If vertex $u$ performs a comparison-based test on vertices $v$ and $w$, then $u$ is referred to as the comparator, while $v$ and $w$ are termed the testees. The outcome of this comparison test is denoted by $\sigma_u(v, w)$, and its value is defined as follows:
$$
\sigma_u(v, w) =
\begin{cases}
0, & \text{if } \{u, v, w\} \cap F = \emptyset; \\
1, & \text{if } u \notin F \text{ and } \{v, w\} \cap F \neq \emptyset; \\
0 \text{ or } 1, & \text{if } u \in F,
\end{cases}
$$
where the outcome is reliable only if the comparator is fault-free (i.e., $u \notin F$).

Based on the above testing rules, the following related definitions can be given.The collection of all test outcomes is called a syndrome, denoted by $\sigma$. If there exists a faulty vertex set $F$ (with the corresponding fault-free vertex set being $V(G) - F$) that can produce this syndrome, then $F$ is said to be compatible with $\sigma$. Denote by $\sigma(F)$ the set of all possible syndromes produced by the faulty set $F$.

Let $F_1$ and $F_2$ be two distinct faulty vertex subsets. If $\sigma(F_1) \cap \sigma(F_2) = \emptyset$, then $F_1$ and $F_2$ are said to be distinguishable; otherwise, they are indistinguishable.

Next, we present two lemmas that provide the necessary and sufficient conditions for two distinct faulty sets \(F_1\) and \(F_2\) to be distinguishable under the PMC model and the MM\(^*\) model, respectively.

\begin {lemma}\label{L2.1}\rm(\cite{dahbura19840}).
For any two distinct faulty sets $F_1, F_2 \subseteq V(G)$, $F_1$ and $F_2$ are distinguishable under the PMC model if and only if there exists an edge $(x,y)$, where the vertex $x\in V(G)\setminus(F_1 \cup F_2)$ and $y\in F_1 \triangle F_2$ (see Fig. \ref{fig:PMC}).
\end {lemma}

\begin{figure}[!htbp]
        \centering
        \includegraphics[width=0.6\linewidth]{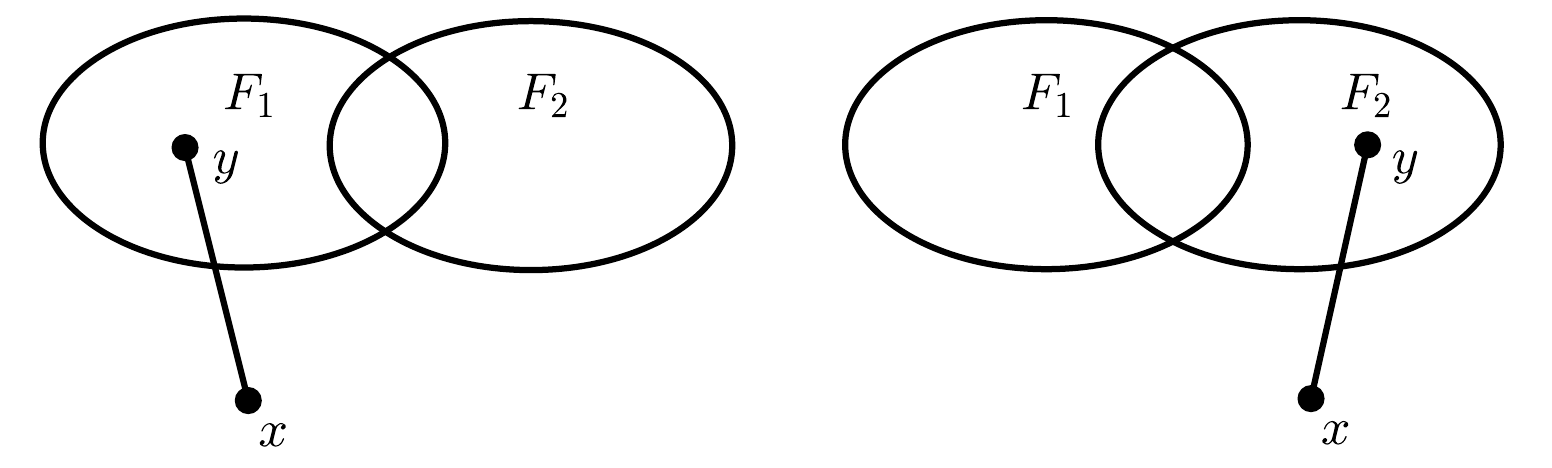}
        \caption{Illustration of the distinguishable sets $F_1$ and $F_2$ under the PMC model.}
        \label{fig:PMC}
\end{figure}

\begin {lemma}\label{L2.2} \rm(\cite{sengupta1992self}).
For any two distinct faulty sets $F_1, F_2 \subseteq V(G)$, $F_1$ and $F_2$ are distinguishable under the MM$^*$ model if and only if one of the following three conditions is satisfied (see Fig. \ref{fig:MM})
\begin{enumerate}
    \item[(1)] There exist vertices $x, y \in V(G) \setminus (F_1 \cup F_2)$ and $z \in F_1 \triangle F_2$ such that $(x, y), (x, z) \in E(G)$.
    \item[(2)] There exist vertices $y, z \in F_2 \setminus F_1$ and $x \in V(G) \setminus (F_1 \cup F_2)$ such that $(x, y), (x, z) \in E(G)$.
    \item[(3)] There exist vertices $y, z \in F_1 \setminus F_2$ and $x \in V(G) \setminus (F_1 \cup F_2)$ such that $(x, y), (x, z) \in E(G)$.
\end{enumerate}
\end {lemma}

\begin{figure}[!htbp]
    \centering
    \includegraphics[width=0.4\linewidth]{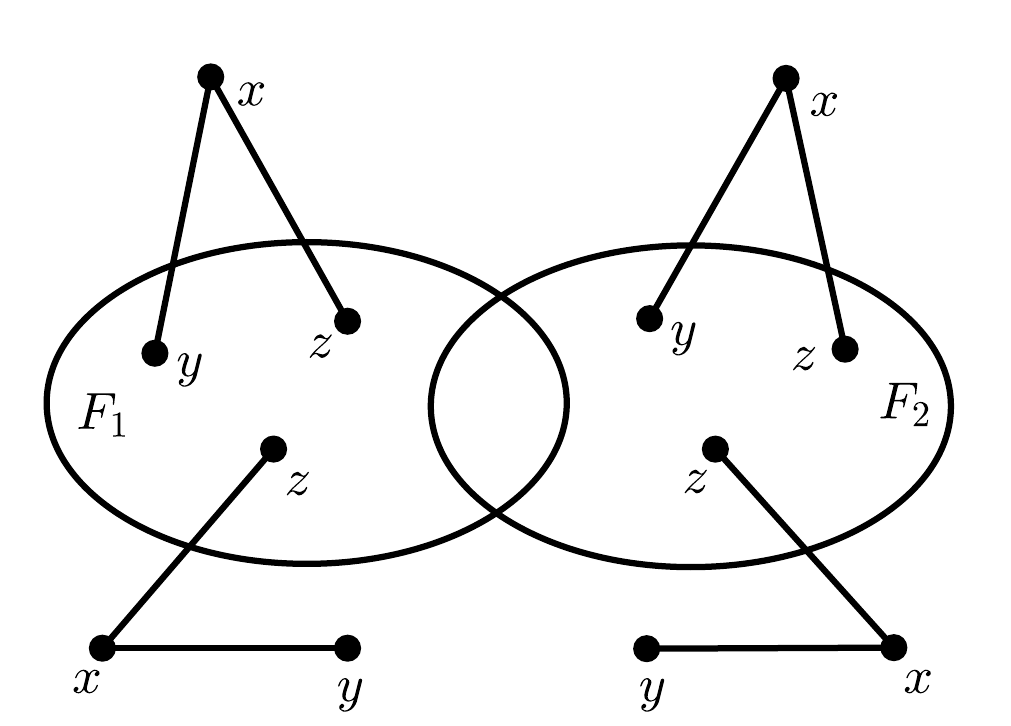}
    \caption{Illustration of the distinguishable sets $F_1$ and $F_2$ under the MM$^*$ model.}
    \label{fig:MM}
\end{figure}

\begin {definition}\label{D2.3} \rm(\cite{10.1093/comjnl/bxad012}).
For a graph $G$, a subset $F\subseteq V(G)$ is called a cyclic vertex subset of $G$ if $G-F$ is disconnected and at least two components in $G-F$ contain a cycle.
\end {definition}

\begin {definition}\label{D2.4} \rm(\cite{zhang2024cyclic}).
A system $G$ is said to be cyclic $t$-diagnosable if and only if for any two distinct faulty cyclic subsets $F_1, F_2 \subseteq V(G)$ with $|F_1| \le t$ and $|F_2| \le t$, $F_1$ and $F_2$ are distinguishable. The cyclic diagnosability of graph $G$, denoted by $ct(G)$, is defined as the maximum value of $t$ such that $G$ is cyclic $t$-diagnosable.
\end {definition}

\subsection{Cayley graphs generated by transposition triangle-free unicyclic graphs}

Let $\Gamma$ be a finite group, and $S \subseteq \Gamma \setminus \{e\}$, where $e$ denotes the identity element of the group. The Cayley digraph $\operatorname{Cay}(\Gamma, S)$ is defined as the digraph with vertex set $\Gamma$ and arc set $\{(g, g \cdot s) \mid g \in \Gamma,\ s \in S\}$, the arc $(g, g \cdot s)$ is said to have label $s$. In particular, if $S^{-1} = S$ , the digraph can be regarded as an undirected graph, called a Cayley graph. In that case, the graph is connected if and only if $S$ is a generating set of $\Gamma$.  Notably, all Cayley graphs are vertex-transitive.

In this paper, we consider Cayley graphs $\operatorname{Cay}(\operatorname{Sym}(n), \tau)$, where $\operatorname{Sym}(n)$ is the group of all permutations in the set $[n] = \{1, 2, \cdots, n\}$ and $\tau$ is a set of transpositions of $\operatorname{Sym}(n)$. To clearly describe the generating structure, we define $G(\tau)$ as the transposition generating graph of $\operatorname{Cay}(\operatorname{Sym}(n), \tau)$: its vertex set is $[n]$, and its edge set is $\{(i, j) \mid (ij) \in \tau,\ i, j \in [n]\}$. The Cayley graph $\operatorname{Cay}(\operatorname{Sym}(n),\tau)$ is connected if and only if $G(\tau)$ is connected. 

For notation, we use $p_1 p_2 \cdots p_n$ to denote the permutation
\[
\begin{pmatrix}
1 & 2 & \cdots & n \\
p_1 & p_2 & \cdots & p_n
\end{pmatrix},
\]
and write $(kl)$ for the transposition that exchanges the entries in positions $k$ and $l$:
\[
(p_1 \cdots p_k \cdots p_l \cdots p_n)\cdot(kl) = p_1 \cdots p_l \cdots p_k \cdots p_n.
\]

If $G(\tau)$ is a tree, then $\operatorname{Cay}(\operatorname{Sym}(n),\tau)$ is called the Cayley graph generated by transposition trees, and denoted by $\mathcal{T}_n$. This constitutes an extremely rich family of graphs. In particular, when $G(\tau)$ is a star graph, $\operatorname{Cay}(\operatorname{Sym}(n),\tau)$ corresponds to the well-known star graph $S_{n}$; when $G(\tau)$ is a path graph, it yields the well-known bubble-sort graph $B_{n}$.  If $G(\tau)$ is a unicyclic graph, then $\operatorname{Cay}(\operatorname{Sym}(n),\tau)$ is called the Cayley graph generated by transposition unicyclic graphs and denoted by $UG_{n}$. The richness of unicyclic graphs implies that $UG_{n}$ also forms an extremely rich family of graphs. In particular, when $G(\tau)$ is a cycle, it corresponds to the well-known modified bubble-sort graph $MB_{n}$ (see Fig. \ref{fig:MB4}).

\begin{figure}[!ht]
    \centering
    \includegraphics[width=0.7\linewidth]{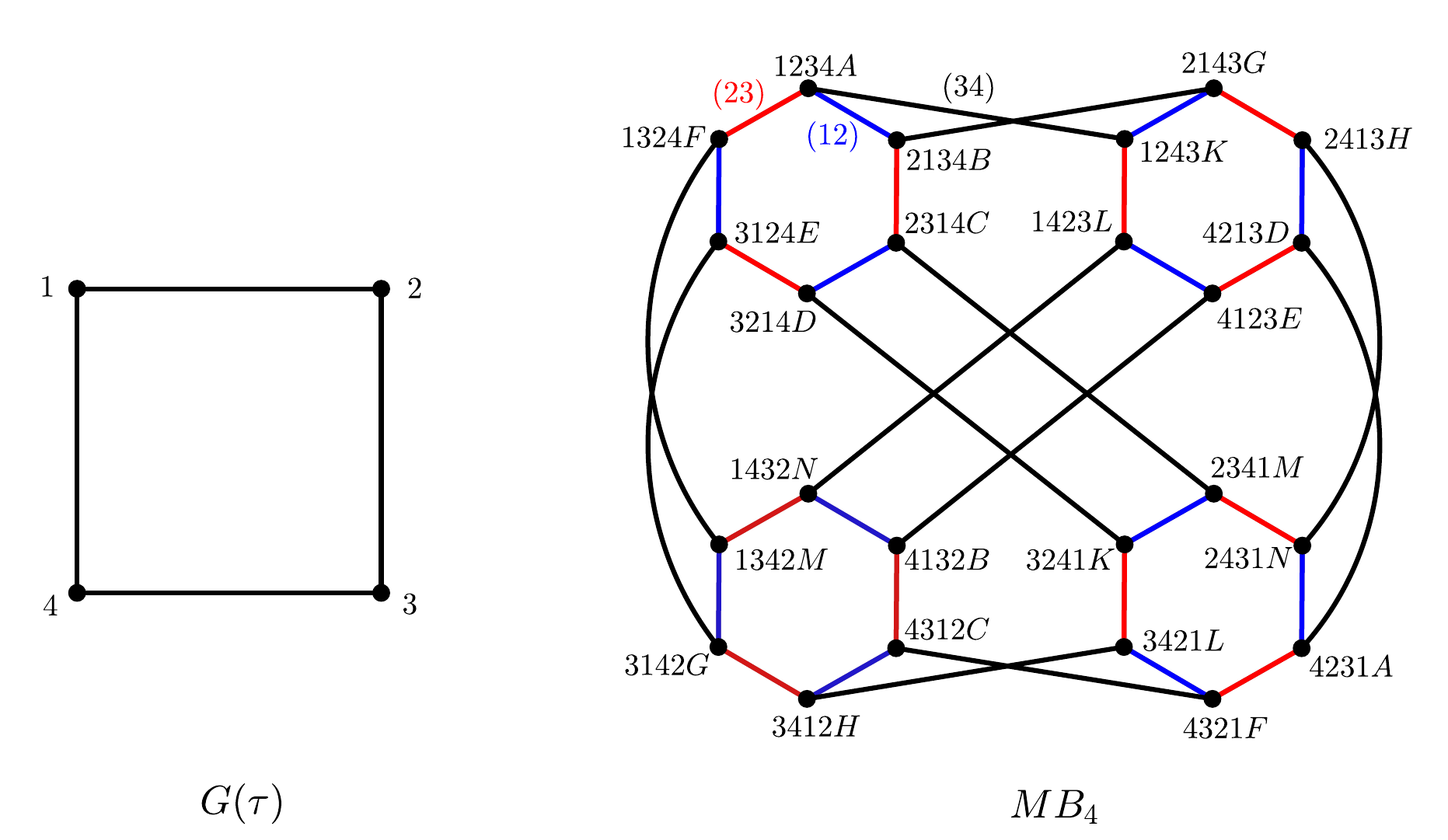}
    \caption{Cayley graph $MB_4$ and its transposition generating graph $G(\tau)$}
    \label{fig:MB4}
\end{figure}

In this paper, we consider the case where $G(\tau)$ is a triangle-free unicyclic graph. The resulting graph is referred to as the Cayley graph generated by transposition triangle-free unicyclic graphs, and is still denoted by $UG_n$. To avoid misunderstanding, we will specifically indicate when $G(\tau)$ is triangle-free. 

\begin {lemma}\label{bipartite}\rm(\cite{LI202095}).
The Cayley graph $UG_n$ is an $n$-regular bipartite graph with $n!$ vertices.
\end {lemma}

\begin {lemma}\label{girth}\rm(\cite{YU201386}).
 The girth of $UG_n$ is 4. Let $m$ be the length of the unique cycle in $G(\tau)$. When $m = 3$, any $4$-cycle $\left \langle u_1, u_2, u_3, u_4 \right \rangle $ in $UG_n$ takes one of the following two forms:
\begin{enumerate}
    \item [(1)] $u_2 = u_1\cdot(ij)$, $u_3 = u_2\cdot(kl)$, $u_4 = u_3\cdot(ij)$, $u_1 = u_4\cdot(kl)$, where $i,j,k,l$ are distinct.
    \item [(2)] $u_2 = u_1\cdot(sp)$, $u_3 = u_2\cdot(pt)$, $u_4 = u_3\cdot(sp)$, $u_1 = u_4\cdot(st)$, where $s,p,t$ form the unique 3-cycle in $G(\tau)$.
\end{enumerate}
When $m \geq 4$, any 4-cycle in $UG_n$ only takes the first form described above.
\end {lemma}

\begin {lemma}\label{neighbor}\rm(\cite{YU201386}).
Let $m$ be the length of the unique cycle in $G(\tau)$. For any two distinct vertices $u, v$ in $UG_n$, the following conclusions hold:
\begin{enumerate}
    \item[(1)] When $m = 3$, $|N_{UG_n}(u) \cap N_{UG_n}(v)| \leq 3$. 
    \item[(2)] When $m \geq 4$, $|N_{UG_n}(u) \cap N_{UG_n}(v)| \leq 2$.
    \item [(3)] If $|N_{UG_n}(u) \cap N_{UG_n}(v)| = 2$, then these two common neighbors can be expressed as $u_1 = u\cdot(ij) = v\cdot(kl)$ and $u_2 = u\cdot(kl) = v\cdot(ij)$, where $i,j,k,l$ are distinct.
    \item[(4)] If $|N_{UG_n}(u) \cap N_{UG_n}(v)| = 3$ (this case only exists when $m = 3$), then these three common neighbors take the form $u_1 = u\cdot(st) = v\cdot(sp)$, $u_2 = u\cdot(sp) = v\cdot(pt)$, $u_3 = u\cdot(pt) = v\cdot(st)$, where $\langle s, p, t\rangle$ is the unique triangle in $G(\tau)$.
\end{enumerate}
\end {lemma}

\begin {lemma}\label{large component 1}\rm(\cite{LI202095}).
Let $UG_n$ be the Cayley graph generated by transposition triangle-free unicyclic graphs, where $1 \leq p \leq n-2$ and $n \geq 4$, and let $F$ be a vertex set of $UG_n$. If $\left| F \right| \leq pn - \frac{p\left( p+1 \right)}{2}$, then $UG_n - F$ contains a large component, and the number of vertices in the remaining small components does not exceed $p-1$.
\end {lemma}


\begin {theorem}\label{odd}\rm(\cite{bondy1976graph}).
A graph $G$ has a perfect matching if and only if $odd(G - S) \leq |S|$ for all $S \subseteq V(G)$.
\end {theorem}

\begin {lemma}\label{matching}\rm(\cite{bondy1976graph}).
If $G$ is a $k$-regular bipartite graph, then $G$ has a perfect matching.
\end {lemma}

\section{Main result}
\label{section2}
\setlength{\parindent}{2em}
In this section,  we focus on the cyclic diagnosability of the Cayley graph generated by transposition triangle-free unicyclic graph, we begin by examining the internal structure of $UG_n$.

\begin{lemma}\label{H}
Let \(UG_n\) denote the Cayley graph generated by transposition triangle-free unicyclic graphs. Let \(u\) be a vertex of \(UG_n\) and let \(C_1, C_2\) be two distinct $4$-cycles such that \(V(C_1) \cap V(C_2) = \{u\}\). Denote by \(v_1\) and \(v_2\) the vertices on \(C_1\) and \(C_2\), respectively, that are not adjacent to \(u\) $($see Fig. $\ref{fig:3.1}$$)$. Then we have \(d(v_1, v_2) \geq 3\).
\end{lemma}
\setlength{\parindent}{2em}
\noindent{\bf Proof.} 
By Lemma \ref{girth}, there exist four different transpositions $(i_1i_2), (i_3i_4)$ and $(j_1j_2), (j_3j_4)$ such that 
$$C_1=\langle u, u\cdot(i_1i_2), u\cdot(i_1i_2)\cdot(i_3i_4), u\cdot(i_3i_4)\rangle, C_2=\langle u, u\cdot(j_1j_2), u\cdot(j_1j_2)\cdot(j_3j_4), u\cdot(j_3j_4)\rangle,$$
where $i_1,i_2,i_3,i_4$ are pairwise distinct and $j_1,j_2,j_3,j_4$ are pairwise distinct. 
It suffices to show that $$d(u\cdot(i_1i_2)\cdot(i_3i_4), u\cdot(j_1j_2)\cdot(j_3j_4))\ge 3.$$ 

\begin{figure}[H]
    \centering
    \includegraphics[width=0.5\linewidth]{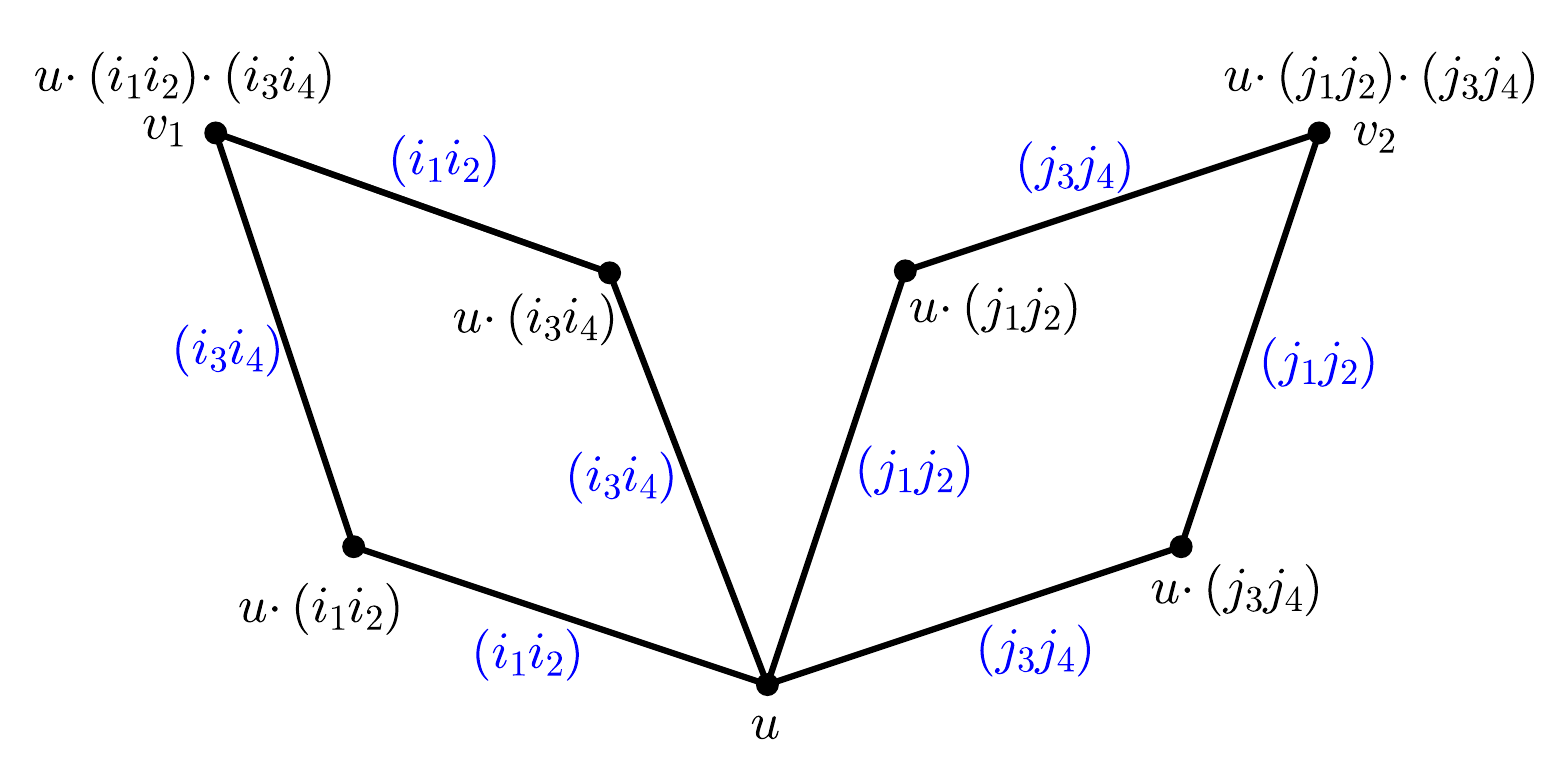}
    \caption{Illustration of two cycles $C_1, C_2$ of Lemma \ref{H}}
    \label{fig:3.1}
\end{figure}

If $|\{i_1, i_2, i_3, i_4\} \cap \{j_1,j_2,j_3,j_4\}|\le 3$, then $u\cdot(i_1i_2)\cdot(i_3i_4)$ and $u\cdot(j_1j_2)\cdot(j_3j_4)$ differ in at least $5$ positions because $\{(i_1i_2), (i_3i_4)\}\cap \{(j_1j_2), (j_3j_4)\}=\emptyset$. Since a single transposition can only alter two positions of the vertex, we know that the distance between $u\cdot(i_1i_2)\cdot(i_3i_4)$ and $u\cdot(j_1j_2)\cdot(j_3j_4)$ is at least 3, as desired.

Otherwise, $\{i_1, i_2, i_3, i_4\} =\{j_1,j_2,j_3,j_4\}$. By the definition of $UG_n$, $(i_1,i_2), (i_3,i_4), (j_1,j_2), (j_3,j_4)$ are four different edges of $G(\tau)$. It follows that the induced subgraph $G(\tau)[i_1,i_2, i_3,i_4]$ is a cycle of length four. Without loss of generality, assume $(j_1j_2)=(i_2i_3)$, $(j_3j_4)=(i_1i_4)$. Note that $G(\tau)$ is triangle-free. We can find $d(u\cdot(i_1i_2)\cdot(i_3i_4), u\cdot(j_1j_2)\cdot(j_3j_4))\ge 3$ (On the contrary, suppose $d(u\cdot(i_1i_2)\cdot(i_3i_4), u\cdot(j_1j_2)\cdot(j_3j_4))=2$, then $v_1 =(v_2 \cdot (i_1i_3)) \cdot (i_2i_4)=(v_2 \cdot (i_2i_4))\cdot (i_1i_3)$. It follows that $(i_1, i_3)$ is an edge of $G(\tau)$. It contradicts the fact $G(\tau)[i_1,i_2, i_3,i_4]$ is a cycle of length four.)

\hfill$\Box$\\

\begin{lemma}\label{A1}
Let \( UG_n \) be a Cayley graph generated by transposition triangle-free unicyclic graphs and $n\geq 5$. Let \( S \) be a subset of \( V(UG_n) \) satisfying \( |S| = 5 \) and \( UG_n[S] \) contains a $4$-cycle. Then, the following two results hold:  
\begin{enumerate}
 \item[{\rm(1)}] If \( UG_n[S] \) is connected, then \( |N_{UG_n}(S)| \geq 5n - 12 \).
 \item[{\rm(2)}] If \( UG_n[S] \) is disconnected, then \( |N_{UG_n}(S)| \geq 5n - 10 \).
 \end{enumerate}
\end{lemma}
\setlength{\parindent}{2em}
\noindent{\bf Proof.}
Let \( S = \{u_1, u_2, u_3, u_4, v\}\), and let \( C = \left \langle   u_1, u_2, u_3, u_4\right \rangle  \) be the induced cycle in \(UG_n[S] \).
\begin{enumerate}
    \item[{\rm(1)}] Suppose \( UG_n[S] \) is connected. In this case, \( UG_n[S] \) is a subgraph containing a $4$-cycle \( C = \left \langle   u_1, u_2, u_3, u_4\right \rangle \) with a pendant vertex \( v \) attached to one of the cycle's vertices. Without loss of generality, assume \( (u_1, v) \in E(UG_n) \). By Lemma \ref{neighbor}(2), vertices \( u_1, u_2, u_3, u_4 \) share no common neighbors outside \( C \),  \( v \nleftrightarrow u_3 \) and \( |N_{UG_n \setminus C}(v) \cap N(u_i)| \leq 1 \) for \( i \in \{2, 4\} \) (see Fig. \ref{fig:3.2}). By Lemma \ref{bipartite},  \( UG_n \) is a bipartite graph. It follows that \( v \nleftrightarrow u_2 \),  \( v \nleftrightarrow u_4 \) and \( N_{UG_n}(v) \cap N_{UG_n}(u_3) = \emptyset \). Thus,   
\[
|N_{UG_n}(S)| \geq 4(n-3) + (n-2) + 2 = 5n - 12.
\]

\begin{figure}[H]
    \centering
    \includegraphics[width=0.5\textwidth]{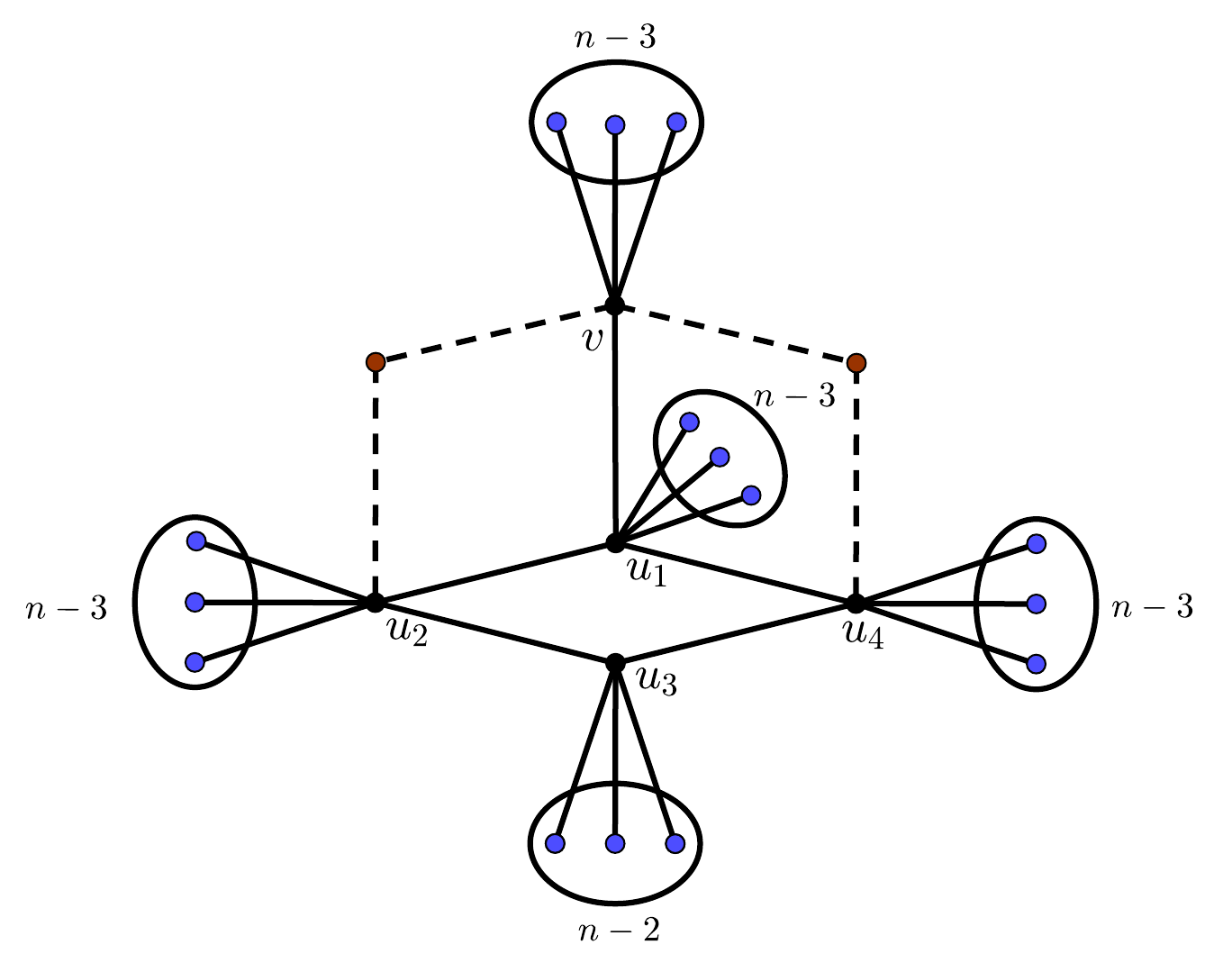}
    \caption{Illustration of the neighborhood of \( UG_n[S] \) when \( UG_n[S] \) is connected in Lemma \ref{A1}}
    \label{fig:3.2}
\end{figure}

\item[{\rm(2)}] Suppose \( UG_n[S] \) is disconnected. In this case, \( N_{UG_n}(v) \cap V(C) = \emptyset \). By Lemma \ref{bipartite}, \( UG_n \) is a bipartite graph. Let $A=\{u_1, u_3, \cdots\}$ and $B=\{u_2, u_4, \cdots\}$ be the two partitions of $UG_n$ and without loss of generality, assume $v\in B$. Then \( v \) shares no common neighbors with \( u_1 \) and \( u_3 \).  By Lemma \ref{neighbor}(2) and Lemma \ref{H}, \( \left| N_{UG_n}(v) \cap N_{UG_n}\left( \{u_2, u_4\} \right) \right| \leq 2 \), which implies \( \left| N_{UG_n}(v) \cap N_{UG_n}(C) \right| \leq 2 \). Therefore,  
\begin{align*}
|N_{UG_n}(S)| &= \left| N_{UG_n}(v) \right| + \left| N_{UG_n}(C) \right| - \left| N_{UG_n}(v) \cap N_{UG_n}(C) \right| - 2\left| N_{UG_n}(v) \cap V(C) \right| \\
&\geq n + (4n - 8) - 2 - 0\\
&= 5n - 10.
\end{align*}
\hfill$\Box$
 \end{enumerate}

\begin{lemma}\label{lemma3.3}
Let \( UG_n \) be the Cayley graph generated by transposition triangle-free unicyclic graphs and $n\geq 5$. Then, under both the PMC model and the MM$^*$ model, \( ct(UG_n) \leq 5n - 10 \).
\end{lemma}
\setlength{\parindent}{2em}
\noindent{\bf Proof.} Note that \( G(\tau) \) is a triangle-free unicyclic graph with at least 5 vertices. It follows that \( G(\tau) \) contains a subgraph isomorphic to \( P_5 \) (not necessarily induced). Without loss of generality, we may assume the edge set of this \( P_5 \) is \( E(P_5) = \{(1,2), (2,3), (3,4), (4,5)\} \). Let \( u = 12345\cdots n \), \(w= (((u\cdot(12))\cdot(23))\cdot(34))\cdot(45)\), and \( v = (u\cdot(23))\cdot(45)\) (see Fig. \ref{fig:4circle}). We construct two 4-cycles as follows: 
\[
C_1 = \left \langle  u, u\cdot(12), u\cdot(12)\cdot(34), u\cdot(34) \right \rangle, \quad C_2 = \left \langle w, w\cdot(12), w\cdot(12)\cdot(34), w\cdot(34)\right \rangle.
\]

Then \( N_{UG_n}(v) \cap N_{UG_n}(C_1) = \{u\cdot(23), u\cdot(45)\} \), and \( C_1 \), \( C_2 \), and \( v \) are pairwise non-adjacent. Note that \( C_1 \) and \( C_2 \) lie in distinct components of \( G - N_{UG_n}(C_1 \cup \{v\}) \). It follows that \( N_{UG_n}(C_1 \cup \{v\}) \) is a faulty cyclic subset of \( UG_n \). Moreover, the size of this neighborhood satisfies: 
\[
\begin{aligned}
|N_{UG_n}(C_1 \cup \{v\})| = 3(n-2) + (n-2-2) + (n-2) + 2= 5n - 10.
\end{aligned}
\]

\begin{figure}[!ht]
    \centering
    \includegraphics[width=0.8\linewidth]{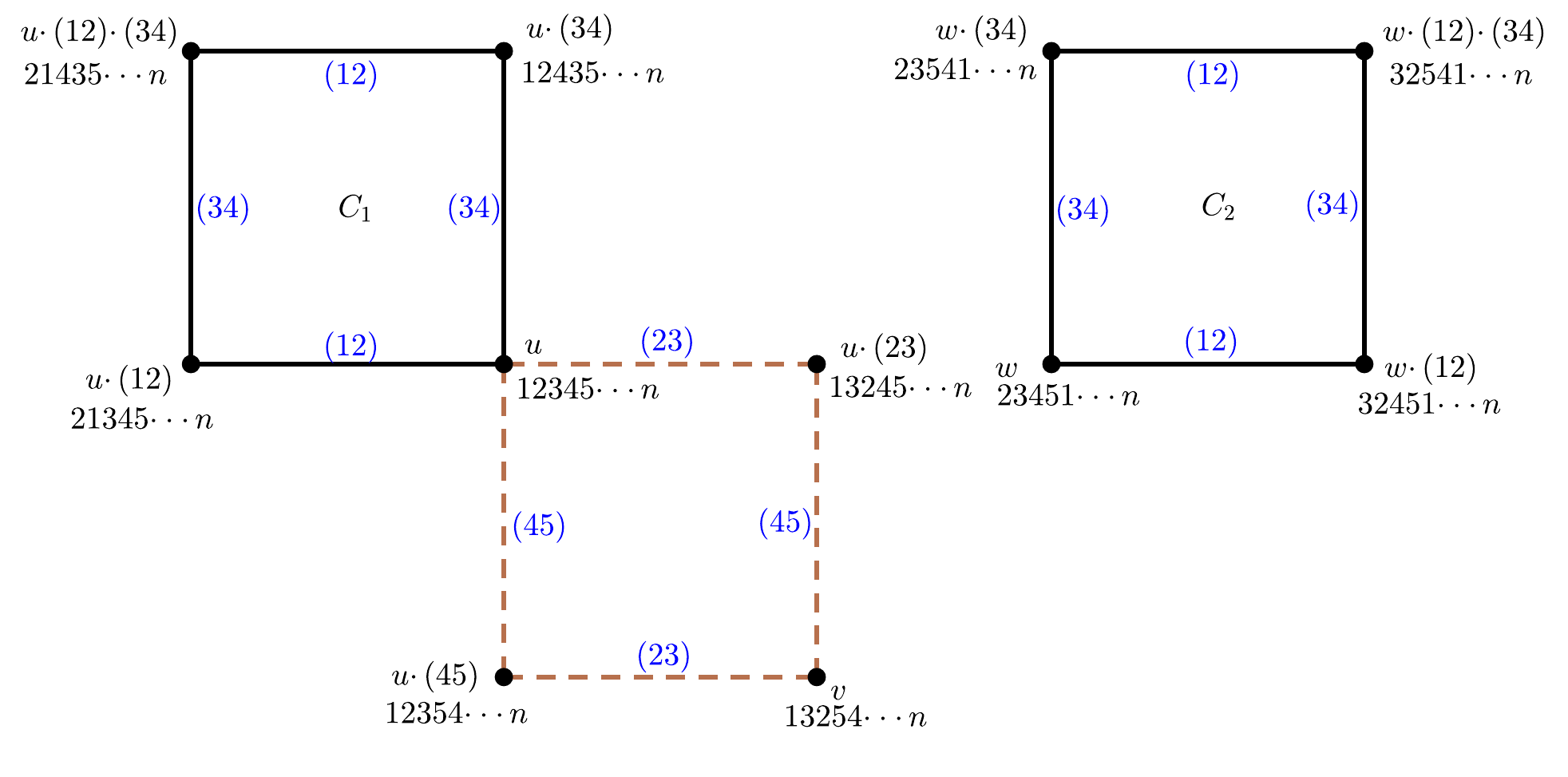}
    \caption{Illustration of two distinct cycles $C_1, C_2$ and the vertex $v$ of Lemma \ref{lemma3.3}}
    \label{fig:4circle}
\end{figure}

Let \( F_1 = N_{UG_n}(C_1 \cup \{v\}) \) and \( F_2 = F_1 \cup \{v\} \). Clearly, \( F_1 \) and \( F_2 \) are both faulty cyclic vertex subsets of \( UG_n \), with \( F_1 \Delta F_2 = \{v\} \) and \( N_{UG_n}(v) \subset N_{UG_n}(C_1 \cup \{v\}) = F_1 \cap F_2 \). Consequently, there are no edges between \( F_1 \Delta F_2 \) and \( \overline{F_1 \cup F_2} \). By Lemmas \ref{L2.1} and \ref{L2.2}, \( F_1 \) and \( F_2 \) are indistinguishable under both the PMC and MM$^*$ models. Therefore, under these models, we have  
\[
ct(UG_n) \leq \max\{|F_1|, |F_2|\} - 1 = |F_2| - 1 = 5n - 10.    
\] \hfill$\Box$

\begin{lemma}\label{>1}
Let $UG_n$ be the Cayley graph generated by transposition triangle-free unicyclic graphs and $n \geq 11$. Then $ct(UG_n) \geq 5n - 10$ under both the PMC model and the MM$^*$ model.
\end{lemma}
\setlength{\parindent}{2em}
\noindent{\bf Proof.}
By Definition \ref{D2.4}, we only need to show that $UG_n$ is $(5n - 10)$-diagnosable under the constraint of faulty cyclic sets. This is equivalent to proving that $F_1$ and $F_2$ are distinguishable for any two faulty cyclic subsets $F_1$ and $F_2$ of $UG_n$ with $\max\{|F_1|, |F_2|\} \leq 5n - 10$.

On the contrary,  suppose that $F_1$ and $F_2$ are indistinguishable. Without loss of generality, assume $F_2\setminus F_1 \neq \emptyset$. Note that 
\begin{align*}
|\overline{F_1\cup F_2}|= |V(UG_n)| - |F_1\cup F_2|\geq |V(UG_n)| - |F_1| - |F_2|\geq n! - 2(5n-10) > 0. 
\end{align*}
It follows that $\overline{F_1\cup F_2} \neq \emptyset$. Note that \( |F_2| \leq 5n - 10 \leq 6n - 21 \) for \( n \geq 11 \). By Lemma \ref{large component 1}, \( UG_n - F_2 \) contains a large component \( B \) and several small components, where the total number of vertices in these small components is at most 5. Let \( M = UG_n - F_2 - B \). As \( F_2 \) is a faulty cyclic subset, both \( B \) and \( M \) contain cycles. By Lemma \ref{girth}, we have \( 4 \leq |V(M)| \leq 5 \). Thus,  
\[
|V(B)| = |V(UG_n)| - |F_2| - |M| \geq n! - (5n - 10) - 5 = n! - 5n + 5.
\]  
Furthermore,  
\[
|V(B)| - |F_1| \geq n! - 5n + 5 - (5n - 10) = n! - 10n + 15 > 0 \tag{$*$}
\]  
implying \( V(B) \setminus (F_1 \setminus F_2) \neq \emptyset \). To precisely characterize the structural relationship between \( B \) and \( F_1 \setminus F_2 \), we state the following claim:

\noindent \textbf{Claim 1.} 
$V(B)\cap(F_1\setminus F_2) = \emptyset$, i.e., $F_1\setminus F_2 \subset V(M)$. 

\noindent {\bf Proof of Claim 1.} First, consider the PMC model. On the contrary, suppose there exists a vertex \( u \in V(B) \cap (F_1 \setminus F_2) \). Since \( B \) is a large component and \( |V(B)| - |F_1| \stackrel{(*)}{>} 0 \), there exists a vertex \( v \in V(B) \setminus (F_1 \setminus F_2) \) such that \( (u, v) \in E(UG_n) \). Note that \( u \in V(B) \cap (F_1 \setminus F_2) \subseteq F_1 \Delta F_2 \), and \( v \in V(B) \setminus (F_1 \setminus F_2) = V(B) \setminus (F_1 \cup F_2) \subset \overline{F_1 \cup F_2} \). By Lemma \ref{L2.1}, \( F_1 \) and \( F_2 \) are distinguishable under the PMC model, contradicting the assumption that $F_1$ and $F_2$ are indistinguishable. Thus, $V(B)\cap(F_1\setminus F_2) = \emptyset$ under the PMC model. 


Next, consider the MM$^*$ model. On the contrary, assume \( V(B) \cap (F_1 \setminus F_2) \neq \emptyset \). Since \( B \) is a component, there exist vertices \( x \in V(B) \cap (F_1 \setminus F_2) \) and \( y \in V(B) \setminus (F_1 \cup F_2) \) such that \( (x, y) \in E(UG_n) \). Define the subset \( A \subseteq B \) where for every vertex \( a \in A \), there exists a vertex \( b \in V(B) \cap (F_1 \setminus F_2) \) with \( (a, b) \in E(UG_n) \)(see Fig. \ref{fig:Lemma3.4-1}). 
If \([ V(B) \setminus (F_1 \cup F_2) ]\setminus A \neq \emptyset \), the connectivity of \( B \) ensures an edge between \([V(B) \setminus (F_1 \cup F_2)] \setminus A \) and \( A \). By Lemma \ref{L2.2}, \( F_1 \) and \( F_2 \) are distinguishable under the MM$^*$ model. This contradicts the assumption that \( F_1 \) and \( F_2 \) are indistinguishable under MM$^*$ model. Thus, we  can assume that \([ V(B) \setminus (F_1 \cup F_2)] \setminus A = \emptyset \). It follows that \( A = V(B) \setminus (F_1 \cup F_2) \). Because \( F_1 \) and \( F_2 \) are indistinguishable under MM$^*$ model, \( A = V(B) \setminus (F_1 \cup F_2) \) is an independent set. It follows that 
\[
|V(B)\setminus(F_1\setminus F_2)| = |V(B)\setminus(F_1\cup F_2)| \leq odd(UG_n-(F_1\cup F_2))\]
\begin{figure}[!ht]
    \centering
    \includegraphics[width=0.5\linewidth]{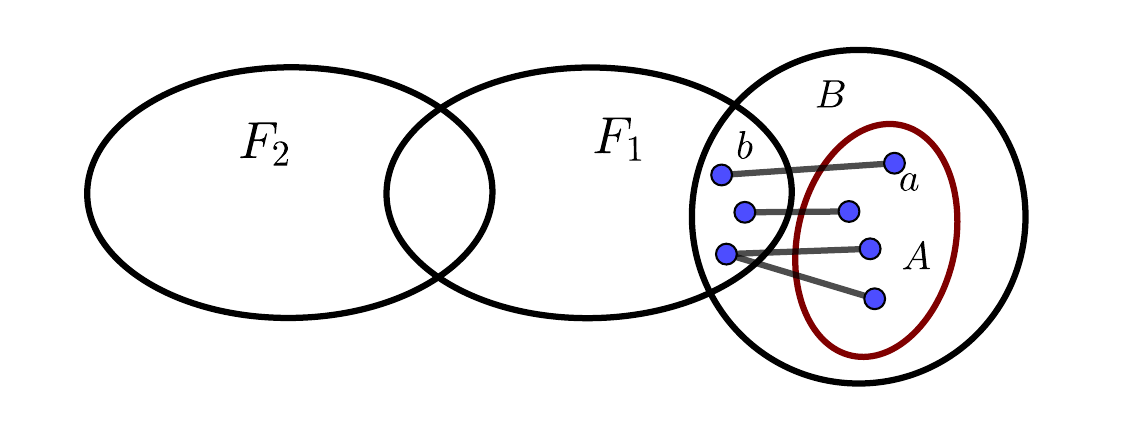}
    \caption{Illustration of the proof of claim 1 under MM$^*$ model in Lemma \ref{>1}}
    \label{fig:Lemma3.4-1}
\end{figure}
By Lemma \ref{bipartite} and \ref{matching}, $UG_n$ has a perfect matching. By Theorem \ref{odd}, we have 
\[
 odd(UG_n-(F_1\cup F_2)) \leq |F_1\cup F_2| \leq 10n-20.
\]
Thus, we have
\[
|V(B)| \leq |V(B)\setminus(F_1\setminus F_2)| + |F_1\setminus F_2| \leq 10n-20 + |F_1| \leq 15n-30,
\]
which contradicts $|V(B)| \geq n! -5n +5$. Hence $V(B)\cap(F_1\setminus F_2) = \emptyset$.      $\qed$

\begin{figure}[!ht]
    \centering
    \includegraphics[width=0.4\linewidth]{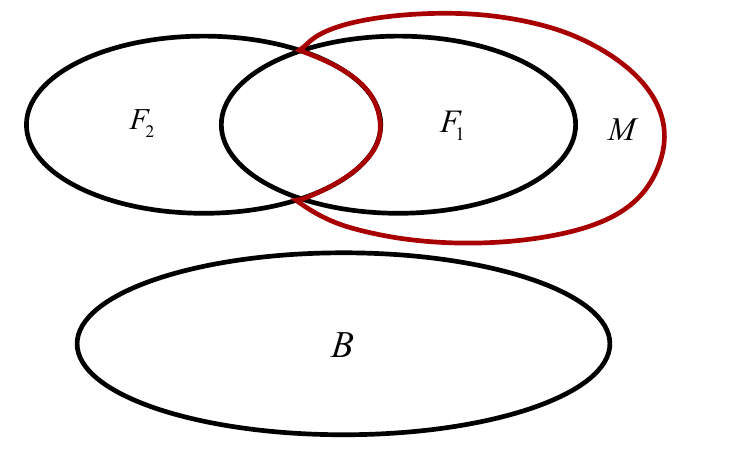}
    \caption{Illustration of structure of $UG_n$ in Lemma \ref{>1}}
    \label{fig:1}
\end{figure}

From Claim 1, we have \( F_1 \setminus F_2 \subset V(M) \) (see Fig. \ref{fig:1}) and thus \( |F_1 \setminus F_2| \leq |M| \leq 5 \). By symmetry, \( 1 \leq |F_2 \setminus F_1| \leq 5 \). Since \( F_1 \) and \( F_2 \) are indistinguishable and \( B \) is a large component, Lemma \ref{L2.1} and \ref{L2.2} imply \( N_{UG_n}(F_2 \setminus F_1) \cap V(B) = \emptyset \). This implies \( UG_n - (F_1 \cap F_2) \) decomposes into three subgraphs: \( B \), \( M \), and \( UG_n[F_2 \setminus F_1] \), where both \( B \) and \( M \) contain cycles. Since  
\[
|V(B)| - |V(M)| - |F_2 \setminus F_1| \geq n! - 5n + 5 - 5 - 5 = n! - 5n - 5 > 0,
\]  
\( B \) remains a large component in \( UG_n - (F_1 \cap F_2) \). Moreover, \( N_{UG_n}(V(M) \cup (F_2 \setminus F_1)) \subset F_1 \cap F_2 \), so \( |N_{UG_n}(V(M) \cup (F_2 \setminus F_1))| \leq |F_1 \cap F_2| \). To further bound the size of \( V(M) \cup (F_2 \setminus F_1) \), we state the following claim: 

{\bf \noindent Claim 2: }\( |V(M) \cup (F_2 \setminus F_1)| \geq 6 \). 
 
{\noindent \bf Proof of Claim 2.} On the contrary, suppose \( |V(M) \cup (F_2 \setminus F_1)| < 6 \). Note that \( 4 \leq |V(M)| \leq 5 \) and \( F_2 \setminus F_1 \neq \emptyset \), so \( |V(M) \cup (F_2 \setminus F_1)| = 5 \). Since \( |F_2 \setminus F_1| \geq 1 \) and \( M \) contains a cycle, Lemma \ref{girth} implies  \( M \) is a 4-cycle and  \( |F_2 \setminus F_1| = 1 \).

\begin{itemize}
    \item Suppose \( F_1 \setminus F_2 \neq \emptyset \). Since \( F_1 \) and \( F_2 \) are indistinguishable and \( M \) is a 4-cycle, we have \( F_1 \setminus F_2 = V(M) \). Recall that \( N_{UG_n}(V(M) \cup (F_2 \setminus F_1)) \subset F_1 \cap F_2 \). It follows that $$ |F_1 \cap F_2| \geq |N_{UG_n}(V(M) \cup (F_2 \setminus F_1))| \geq 5n - 12,$$ where the last inequality is by Lemma \ref{A1}. Thus,  
\[
|F_1| = |F_1 \setminus F_2| + |F_1 \cap F_2| = |M| + |F_1 \cap F_2| \geq 5n - 8,
\]  
contradicting \( |F_1| \leq 5n - 10 \).
\item Suppose \( F_1 \setminus F_2 = \emptyset \), so \( F_1 \subset F_2 \). Since \( F_1 \) and \( F_2 \) are indistinguishable, we have \( N_{UG_n}(F_2 \setminus F_1) \cap V(M) = \emptyset \). By Lemma \ref{A1}, \( |F_1 \cap F_2| \geq |N_{UG_n}(V(M) \cup (F_2 \setminus F_1))| \geq 5n - 10 \). Thus,  
\[
|F_2| = |F_2 \setminus F_1| + |F_2 \cap F_1| \geq 1 + 5n - 10 = 5n - 9,
\]  
contradicting \( |F_2| \leq 5n - 10 \).      $\qed$
\end{itemize}

Next, we complete the proof of this lemma. By Claim 2, $|V(M)\cup(F_2\setminus F_1)| \geq6$. By Lemma \ref{large component 1}, $|F_1\cap F_2| \geq6n-20$, so we have
\[
|F_2| \geq |F_2\setminus F_1| + |F_1\cap F_2| \geq6n-20 >5n-10,
\]
a contradiction to the assumption.

In conclusion, $ct(UG_n) \geq5n - 10$ under both the PMC model and the MM$^*$ model when $n\geq 11$.
\hfill$\Box$\\

Combining with Lemma \ref{lemma3.3} and \ref{>1}, we have the following result.

\begin{theorem}\label{main1}
Let $UG_n$ be the Cayley graph generated by transposition triangle-free unicyclic graphs and $n \geq 11$. Then $ct(UG_n)=5n - 10$ under both the PMC model and the MM$^*$ model.
\end{theorem}

\section{Conclusion}

In this paper, we establish that under the PMC and MM$^*$ models, the cyclic diagnosability of the Cayley graph generated by a triangle-free unicyclic graph is $5n-10$ for $n \geq 11$. In the future, we will conduct in-depth research on the Cayley graphs generated by transposition unicyclic graphs with a triangle.


\bibliographystyle{abbrv}
\bibliography{reference}

\end{document}